\documentclass[11pt]{amsart}

\usepackage{amsmath, amsthm, amssymb, euscript, enumerate}
\usepackage{pxfonts, graphicx, subfigure, epsfig}
\usepackage{mathtools, ulem, appendix, accents, color}

\usepackage{amssymb,mathrsfs,graphicx,enumerate}

\newcommand{\R}{{\mathbb R}}
\newcommand{\pR}{{\mathbb{R}_+^{n+1}}}

\newcommand{\oR}{{\mathbb{R}^{n}}}

\newcommand{\cA}{{\mathcal A}}

\newcommand{\nablas}{{\nabla_s^*}}

\theoremstyle{plain}
\newtheorem{theorem}{Theorem}[section]
\newtheorem{remark}[theorem]{Remark}
\newtheorem{definition}[theorem]{Definition}

\newtheorem{corollary}[theorem]{Corollary}

\newtheorem{lemma}[theorem]{Lemma}
\newtheorem{proposition}[theorem]{Proposition}

\numberwithin{equation}{section}

\title{Liouville-Type Theorems on the Hyperbolic Space}

\author{Sanghoon Lee }
\address{School of Mathematics, Korea Institute For Advanced Study, 85 Hoegiro, Dongdaemungu, Seoul, 02455, South Korea }
\email{sl29@kias.re.kr}

\begin{document}
\let\thefootnote\relax\footnotetext{This work was supported by KIAS Individual Grant (Grant No. MG096101).}


\maketitle
\begin{abstract}

In this paper, we establish Liouville-type theorems for a one-parameter family of elliptic PDEs on the standard upper half-plane model of the hyperbolic space, under specific geometric assumptions. Our results indicate that the Euclidean half-plane is the only compactification of the hyperbolic space when the scalar curvature of the compactified metric has a designated sign.

\end{abstract}

\section{Introduction}
\subsection{Background and setup}
Let $X^{n+1}$ be a  smooth $(n+1)$-dimensional manifold with boundary $\partial X = M$. The interior of $X$ is denoted as $\mathring{X}$. Consider  a smooth boundary defining function $\rho$ that satisfies the following properties: 
\begin{equation*}
0 \le \rho \in C^\infty (X),\, \rho>0 \text{ in } \mathring{X},\, \rho = 0 \text{ on } M,\, D\rho \ne 0 \text{ on } M.
\end{equation*}
by definition.
Given a conformal class of Riemannian metrics $[h]$, we define $(X, \partial X, g_+)$ as a conformally compact Einstein (CCE) manifold  with conformal infinity $(M, [h])$ if $g_+$ is a complete metric on $\mathring{X}$ such that
\begin{equation*}
\begin{cases}
Ric_{g_+} = - n g_+ \\
(\rho^2g_+ )|_{M} \in [h].
\end{cases}
\end{equation*}
for a defining function $\rho$.

The adapted compactification of the metric $g_+$, denoted as $g_s^* \coloneqq \rho_s^2 g_+$, where $s$ is a real parameter,  is   commonly used in the study of CCE manifolds. The conformal factor $\rho_s$ was introduced by Case and Chang in \cite{CC} and is derived by applying scattering theory. For the range $\frac{n}{2} < s \le n+1$ and $s \ne n$, the conformal factor $\rho_s$ is defined as $\rho_s \coloneqq v_s^{\frac{1}{n-s}}$ where $v_s$ is a positive function that satisfies an elliptic PDE with Dirichlet boundary data:
\begin{equation*}
\begin{cases}\label{eq:cce}
-\Delta_+ v_s - s(n-s)v_s = 0  \text{ on } \mathring{X} \\
(\rho_s ^2g_+)|_{M} = (v_s^{\frac{2}{n-s}} g_+ )|_M = h \text{ on } M.
\end{cases}
\end{equation*}


In recent decades, numerous studies have been dedicated to addressing geometric problems of CCE manifolds, including issues related to existence and uniqueness. For more information on these topics, readers are referred to works such as \cite{A, BH, CDLS, CLW, DM, FG2, GH, GHS, GL, GS, LSQ}, to name just a few. This article is particularly relevant to the recent compactness results achieved by Chang and Ge \cite{CG}, Chang-Ge-Qing \cite{CGQ}, and Chang-Ge-Jin-Qing \cite{CGJQ}.

To control the curvature of adapted metrics near the boundary for a sequence of CCE manifolds, sophisticated blow-up analysis techniques are employed. This process results in a scenario involving compactification on the standard upper half-plane model of hyperbolic space, where the boundary has the standard Euclidean metric. In this setting, the limiting conformal factor, defined on the upper half-plane and providing such compactification, satisfies a PDE resembling equation (\ref{eq:cce}) with appropriate boundary conditions.  Thus, establishing Liouville-type theorems for these PDEs becomes an essential step towards achieving compactness results, as well as in gaining a deeper understanding of the boundary geometry of CCE manifolds.

We note that the specific case when $s = \frac{n+3}{2}$ has been previously addressed in the work of Chang and Ge \cite{CG} by using a Liouville-type theorem for harmonic functions.

Our goal is to prove Liouville-type theorems for adapted metrics on the upper half-plane model of the standard hyperbolic space $(\pR, \oR, g_+)$ for a broader range of the parameter $s$, while specifying the sign of the scalar curvature. We denote 
$\R_+^{n+1} \coloneqq \{ z= (x, y) \in \R^{n+1} | x \in \R^n, y>0 \}$, $\oR \coloneqq \{ z= (x, 0) \in \R^{n+1} | x \in \R^n \}$ and $g_+ \coloneqq \frac{1}{y^2}|dz|^2$. Since scattering theory does not apply in the non-compact manifolds case, the adapted metric  for $(\pR, \oR, g_+)$ is defined by $g_s^* \coloneqq \rho_s^2 g_+$, where $\rho_s \coloneqq v_s^{\frac{1}{n-s}}$ for a positive solution $v_s$ of the following PDE:
\begin{equation}
\begin{cases}
-\Delta_+ v_s - s(n-s)v_s = 0 \text{ in } \pR \\
\lim_{y \rightarrow 0} \frac{v_s}{y^{n-s}} = 1 \text{ on } \oR \iff g_s^*|_{\oR} \equiv |dx|^2.
\end{cases}
\end{equation}
We aim to prove that $g_s^* \equiv |dz|^2$, assuming a certain sign for the scalar curvature of the metric $g_s^*$. This assumption about the scalar curvature's sign is supported by the observed behavior of scalar curvature in CCE manifolds when their conformal infinities have positive scalar curvature. It is important to note that in the absence of this assumption on scalar curvature, $v_s = y^{n-s} + \beta y^s$ could represent other solutions for any positive real number $\beta$.

Let us define $s \coloneqq \frac{n}{2} + \gamma$, and $a \coloneqq 1-2\gamma$ for $0<\gamma < \frac{n}{2}+1$.  We then define $w_s \coloneqq 1-\frac{v_s}{y^{n-s}}$ which leads to the expression $v_s = y^{n-s}(1-w_s)$. The equation $-\Delta_+ v_s - s(n-s)v_s = 0$ on $(\pR, \oR, g_+)$ can be rewritten with respect to the Euclidean metric by using the conformal translation rule for the Laplacian. This transformation yields the following equation:

\begin{equation} \label{eq:degenerate1}
\Delta_a w_s = 0
\end{equation}
where the degenerate elliptic operator $\Delta_a$ is defined as $\Delta_a u \coloneqq \mathrm{div} ( y^a \nabla u)$ for a function $u$. We assume that $w_s \equiv 0$ on $\oR$, which is equivalent to stating that $g_s^* \coloneqq \rho_s^2 g_+= v_s^{\frac{2}{n-s}} g_+$ satisfies $g_s|_{\oR} \equiv |dx|^2$. With this boundary condition and an assumption regarding the sign of the scalar curvature, we prove Liouville-type theorems for the PDE (\ref{eq:degenerate1}).

\subsection{Main results}

The first main result of this article examines the sign of the scalar curvature of adapted metrics of CCE manifolds in the case $\gamma \in (0, 1)$. We prove that the sign of the scalar curvature is the same as the sign of $\gamma -\frac{1}{2}$ when the conformal infinity has positive scalar curvature.  Previous work in \cite{CC} demonstrated that the scalar curvature of the adapted metric is positive  for $\gamma \in [1, \frac{n}{2}+1]$ if the scalar curvature of the conformal infinity is positive. Our result extends this finding to $\gamma \in (0, \frac{n}{2}+1]$. We denote the covariant derivative and scalar curvature of the adapted metric $g_s^*$  by $\nablas$ and $R_s^*$.

\begin{theorem} \label{mainthm2.1}
Let $(X^{n+1}, \partial X = M , g_+)$ be a CCE manifold with conformal infinity $(M, [h])$. Assume that the scalar curvature $R_h$ of $(M, h)$ is positive. For $0<\gamma<1$, it follows that $|\nablas \rho_s|^2 < 1$ on $\mathring{X}$, or equivalently, $\mathrm{sgn}(R_s^*) = \mathrm{sgn}(\gamma - \frac{1}{2})$.
\end{theorem}

The aforementioned theorem validates the assumption $|\nablas \rho_s | \le 1$ for the adapted metrics $g_s^*$ on  $(\pR, \oR, g_+)$, and the sign of $R_s^*$ is deduced from the identity $R_s^*  = 2n(s- \frac{n+1}{2}) \cdot \frac{(1-|\nablas \rho_s|^2)}{\rho_s^2}$. Assuming  $|\nablas \rho_s | \le 1$, we have the following Liouville-type theorem for $\gamma \in (0,1)$  which is a consequence of the work by Sire, Terracini, and Vita \cite[Theorem 3.1]{STV2}.
\begin{theorem} \label{mainthm2.2}
Let $\gamma \in (0,1)$. Suppose $v_s$ is a positive solution to the PDE $-\Delta_+ v_s -s(n-s)v_s = 0$ on $\pR$.   Define $w_s \coloneqq 1-\frac{v_s}{y^{n-s}}$. Assume the following conditions:

(1)  $w_s \in \tilde{H}^{1,a}(\pR)$.

(2) $|\nablas \rho_s | \le 1$. This is equivalent to $\mathrm{sgn}(R^*_s) = \mathrm{sgn}(\gamma - \frac{1}{2})$.

Then, $v_s = y^{n-s}$ or equivalently, $g_s^* \equiv |dz|^2$.
\end{theorem}

The precise definition of the functional space $\tilde{H}^{1,a}(\pR)$ is given in the Definition \ref{def:wfunctional}. In essence, $w \in \tilde{H}^{1,a}(\pR)$ if the $y^a$-weighted $W^{1,2}$-Sobolev norm of $w$ is locally bounded and the trace of $w_s$ on $\oR$ is zero.  For $\gamma \in [\frac{1}{2}, \frac{n}{2})$, we prove the following Liouville-type theorem under condition that the scalar curvature of the adapted metric is non-negative.

\begin{theorem} \label{mainthm:1}
Let $s \coloneqq \frac{n}{2} + \gamma$ for some $\gamma \in [\frac{1}{2}, \frac{n}{2})$. Suppose $v_s$ is a positive solution to the PDE $-\Delta_+ v_s -s(n-s)v_s = 0$  on $\pR$. Define $w_s \coloneqq 1-\frac{v_s}{y^{n-s}}$. Assume the following conditions:

(1) $w_s \in \tilde{H}^{1,0}(\pR)$.

(2) $R^*_s $ is non-negative.

Then, $v_s = y^{n-s}$ or equivalently, $g_s^* \equiv |dz|^2$.
\end{theorem}

While proving Theorem \ref{mainthm:1}, we found a Liouville-type theorem(refer to Theorem \ref{thm:liouville}) for non-negative subharmonic functions defined on $\overline{\mathbb{R}}_+^{n+1}$, which, to the best of the author's knowledge, was previously unrecognized. This implies that the conformal factor does not necessarily have to be a solution to a PDE, although being a solution can provide certain advantages in terms of establishing regularity properties. We articulate this observation in the following statement:

\begin{theorem}
Let $\rho$ be a positive function defined on $(\pR, g_+ \coloneqq \frac{1}{y^2}|dz|^2)$ with following properties:

(1) $\rho \in C^2(\overline{\mathbb{R}}_+^{n+1})$

(2) $\rho(x, 0) = 0 $ and $\partial_y \rho(x, 0) = \lim_{y\rightarrow 0 }\frac{\rho}{y}(x, y) = 1$, or equivalently, the induced boundary metric on $\oR$ of the compactified manifold $(\overline{\mathbb{R}}_+^{n+1}, g^* \coloneqq \rho^2 g_+)$ is the Euclidean metric.

(3) The scalar curvature of the compactified metric, $R^*$, is non-negative.

(4) $|\nabla^* \rho| \le 1$.

Then, $\rho \equiv y$, or equivalently, $g^* \equiv |dz|^2$.

\end{theorem}

A direct consequence of the above Theorem  is  Liouville-type theorem for adapted metrics with $\gamma \in [\frac{n}{2}, \frac{n}{2}+1]$, including the Fefferman-Graham compactification.

\subsection{Organization of the paper}
We breifly outline the organization of this paper. In Section 2, we revisit some fundamental results from scattering theory on CCE manifolds and cover basic analytic preliminaries for $y^a$-weighted Sobolev spaces. In section 3, we define the weak solutions of a one-parameter family of degenrate PDEs on the upper half-plane and prove a Liouville theorem for non-negative subharmonic functions defined in this domain.

 In Section 4, we explore geometric properties of CCE manifolds with adapted metrics, particularly when $0<\gamma<1$ and the conformal infinity has positive scalar curvature. We then prove the first main theorem, which determines the sign of the scalar curvature of the adapted metric. In Section 5, we establish Liouville-type Theorems for the adapted metrics. Finally, in Section 6, we give an upper bound on the weighted Sobolev norm of the conformal factors that define the adapted metrics.

\section*{Acknowledgement}
The author wishes to express sincere gratitude to Professor Sun-Yung A. Chang for introducing him to the subject and for providing ongoing encouragement. Special thanks are also extended to Professor Yuxin Ge for the enlightening discussions.  


\section{Preliminaries}

In this section, we discuss pertinent facts from geometry and analysis that are essential for the content of this article.
\subsection{Scattering theory on CCE manifolds and adapted metrics}

Given the vast and comprehensive nature of scattering theory, our focus here is limited to aspects that are directly relevant to our research.

This subsection introduces the concept of adapted metrics in the context of scattering theory on CCE manifolds.  For more comprehensive discussions, readers are directed to references \cite{GZ} and \cite{MM}. In this paper, $s$ and $\gamma$ are real numbers such that $s = \frac{n}{2} + \gamma$.

We begin by revisiting the fact that for a CCE manifold $(X, M, g_+)$ with conformal infinity $(M, [h])$, there exists a defining function $y$ such that
$$g_+ = y^{-2}(dy^2 + g_y), \text{ with } g_0 = h $$
in a collar neighborhood of $M$.  In this expression, $g_y$ represents a metric on the level sets of $y$, which are diffeomorphic to $M$ when $y$ is sufficiently small. This $y$ is called the geodesic defining function. Other literature might denote this function by $x$ or $r$, but we use $y$ as it also represents a geodesic defining function on $\pR$. In this paper, $r$ typically denotes the radius of geodesic balls.

For $s \in \mathbb{C}$ such that $s(n-s) \notin \sigma_{pp}(-\Delta_+)$ and $f \in C^\infty(M)$, there exists  a solution $\mathcal{P}(s)f$ to the Poisson equation
$$(-\Delta_+  - s(n-s))(\mathcal{P}(s)f) = 0 \, \text{ in } X$$
taking the form
$$\begin{cases}
\mathcal{P}(s)f = y^{n-s} F + y^s G \,\,\text{ if }\, s \notin n/2 + \mathbb{N} \\
\mathcal{P}(s)f = y^{n/2-l/2} F + G y^{n/2+l/2} \log y  \,\,\text{ if }\, s = n/2 +l, \, l \in \mathbb{N}
\end{cases}$$
where $F, G \in C^{\infty} (\overline{X})$ and $F|_M = f$.
The operator $\mathcal{P}(s)$ is known as the Poisson operator and is meromorphic in the region $\{\, \mathrm{Re}\, s > n/2 \,\}$, having poles only for $s$ such that $s(n-s) \in \sigma_{pp}(-\Delta_+)$. The scattering operator, denoted $S(s)$, is defined by $S(s) f= G|_M$. The fractional GJMS operator is defined as
$$P(\gamma) \coloneqq 2^{2\gamma} \frac{\Gamma(\gamma)}{\Gamma(-\gamma)} S(\frac{n}{2} + \gamma )$$
where $\Gamma$ represents the Gamma function. Furthermore, it is established that $F$ has an aymptotic expansion $F = f + f_{(2)} y^2 + f_{(4)} y^4 + \cdots$ near $M$.

The fractional $Q_{2\gamma}$ curvature on the boundary is defined as $Q_{2\gamma} \coloneqq \frac{2}{n-2\gamma} P_{2\gamma}(1)$. For more detailed exploration of the conformal fractional Laplacian and fractional $Q_{2\gamma}$ curvature, we direct readers to \cite{CG2} and \cite{CC}. 

We now introduce a family of compactifications, each parametrized by a real parameter $s \coloneqq \frac{n}{2} + \gamma $, for a given CCE manifold $(X, M, g_+)$.
\begin{definition}
For $s>\frac{n}{2}$, $s \ne n$, and when $s(n-s) \notin \sigma_{pp}(-\Delta_+) $, define $v_s \coloneqq \mathcal{P}(s)1$. According to  \cite[Lemma 6.1]{CC}, It is established that $v_s>0$. Let $\rho_s = v_s^{\frac{1}{n-s}}$. The resulting  compactified metric $g_s^* \coloneqq \rho_s^2 g^+$ is termed the adapted metric. 

For the case $s = n$, define $v \coloneqq  -\frac{d}{ds}\big|_{s=n} \mathcal{P}(s)1$. By \cite[Theorem 4.1]{FG}, $v$ satisfies the PDE $-\Delta_+ v = n$. Additionally, the function $v-\log y$ is within $C^3(X)$ and vanishes on $M$. The metric $g^* \coloneqq e^{2v} g_+$ is known as the Fefferman-Graham compactified metric. 
\end{definition}

When the conformal infinity $(M, h)$ has positive scalar curvature. it is known that $-\Delta_+$  has no pure point specturm. Consequently, this ensures that the adapted metric is well-defined for $s > \frac{n}{2}$.
\begin{theorem}\cite[Theorem A]{L}
Let $(X, M, g_+)$ be a CCE manifold with conformal infinity $(M, [h])$. If the Yamabe constant on the boundary is non-negative, then $-\Delta_+$ has no $L^2$ eigenfunctions.
\end{theorem}

\subsection{$y^a$-weighted Sobolev spaces}
Let $a$ be a real number. In this subsection, we explore some analytic foundations of $y^a$-weighted Sobolev spaces on $\R_+^{n+1}$. We denote $B_r^+$ as the half-ball $\{z = (x, y) \in \pR |\, |z|<r \}$, and $(\partial B_r)^+$ as the boundary of this half-ball $\{z = (x, y) \in \pR |\, |z|= r \}$  for any $r>0$.
\begin{definition}\label{def:wfunctional}
For a real number $a$, we define the following Sobolev spaces:
\begin{align*}
H^{1, a} (\pR) \coloneqq & \{ u \in W^{1,1}_{\mathrm{loc}} ( \pR) |\int_{B_r^+} (|\nabla u|^2 + u^2 ) y^a dz <\infty \,  \text{ for every } r>0 \}\\
\tilde{H}^{1, a}(\pR) \coloneqq & \{ u \in H^{1, a}( \pR) |\, \text{For every $r>0$}, \exists  \{\phi_k\}_{k \ge 1},  \text{a sequence of functions }\\
&  \text{in $ C_c^1 (\overline{B_r^+}\setminus  \oR)$ s.t.} \int_{B_r^+} (|\nabla (u-\phi_k)|^2 + |u-\phi_k|^2 ) y^a dz \rightarrow 0  \}.	
\end{align*}
\end{definition}

We have the following inclusion property for $\tilde{H}^{1, a} (\pR)$.
\begin{proposition}\label{lemma:inclusion}
Let $a \le b$ be real numbers. Then, it holds that
$$\tilde{H}^{1, a} (\pR) \subset \tilde{H}^{1, b} (\pR).$$
\end{proposition}
\begin{proof}
Let $u \in \tilde{H}^{1, a} (\pR)$. For a $r>0$,  there exists a sequence of functions $\{\phi_k \in C_c^1 (\overline{B_r^+}\setminus  \oR)\}_{k \ge 1}$ such that $$\int_{B_r^+} (|\nabla (u-\phi_k)|^2 + |u-\phi_k|^2 ) y^a dz \rightarrow 0 $$ by definition. We now check that $\int_{B_r^+} (|\nabla u|^2 + u^2 ) y^b dz <\infty$, and $\int_{B_r^+} (|\nabla (u-\phi_k)|^2 + |u-\phi_k|^2 ) y^b dz \rightarrow 0 $. This will confirm that $u \in  \tilde{H}^{1, b} (\pR)$.
We compute:
\begin{align*}
\int_{B_r^+} (|\nabla u|^2 + u^2 ) y^b dz & = \int_{B_r^+\cap\{y\le 1\}} (|\nabla u|^2 + u^2 ) y^b dz +  \int_{B_r^+ \cap \{1 \le y \le r \}} (|\nabla u|^2 + u^2 ) y^b dz \\
& \le \int_{B_r^+\cap\{y\le 1\}} (|\nabla u|^2 + u^2 ) y^a dz +  r^{b-a} \int_{B_r^+ \cap \{1 \le y \le r \}} (|\nabla u|^2 + u^2 ) y^a dz \\
&\le \max (1, r^{b-a}) \int_{B_r^+} (|\nabla u|^2 + u^2 ) y^a dz <\infty
\end{align*}
Similarly, we observe:
\begin{align*}
\int_{B_r^+} (|\nabla (u-\phi_k)|^2 + |u-\phi_k|^2 ) y^b dz \le \max (1, r^{b-a}) \int_{B_r^+} (|\nabla (u-\phi_k)|^2 + |u-\phi_k|^2 ) y^a dz.
\end{align*}
Since the right-hand side converges to 0, we conclude that $u$ indeed belongs to $\tilde{H}^{1, b} (\pR)$, thus proving the proposition.
\end{proof}

When $a \in (-1, 1)$, we make a straightforward observation: any function that is a member of $\tilde{H}^{1, a} (\pR) $ has a zero trace on $\oR$.
\begin{proposition}\label{lemma:trace}
Assume $a \in (-1, 1)$. If $u \in \tilde{H}^{1, a} (\pR)$, then $u \in  W^{1,1}(B_r^+)$ for any $r>0$ and $u|_{\oR} \equiv 0$.
\end{proposition}
\begin{proof}
By Cauchy-Schwarz inequality, we have:
$$\big(\int_{B_r^+} (|\nabla u| + |u|) \big)^2 \le \int_{B_r^+}2y^a(|\nabla u|^2 + u^2)  \cdot \int_{B_r^+} y^{-a} $$
and the term on the right-hand side is finite given that $a \in ( -1, 1)$.
Using the Cauchy-Schwarz inequality once more, we find that:
$$\big(\int_{B_r^+} (|\nabla(u-\phi_k)| + |u-\phi_k|) \big)^2 \le \int_{B_r^+}2y^a(|\nabla (u-\phi_k)|^2 + (u-\phi_k)^2)  \cdot \int_{B_r^+} y^{-a}  \rightarrow 0$$
for a sequence $\{\phi_k \in C_c^1 (\overline{B_r^+}\setminus  \oR)\}_{k \ge 1}$ such that $$\int_{B_r^+} (|\nabla (u-\phi_k)|^2 + |u-\phi_k|^2 ) y^a dz \rightarrow 0.$$
Consequently, the trace Sobolev inequality ensures that $u$ is zero on the boundary.
\end{proof}

For $a \in (-1, 1)$, the weight $y^a$ belongs to the $\cA_2$-Muckenhoupt class, which means that certain $\cA_2$-weighted inequalities, such as Sobolev embedding theorem and Hardy's inequality, are well-established. The following Hardy-type inequalities, which extend to cover  the range $a \in (-\infty, 1)$, are proved and used in \cite{STV1} and \cite{STV2}. 
\begin{proposition} \cite[Lemmga B.5]{STV2} \label{lemma:whardy}
Let $a \in (-\infty, 1)$ and $r>0$. For $w \in C_c^1(\bar{B}_r^+\setminus \oR)$, we have following inequalities: 

(1) $C_a \int_{ B_r^+} y^{a-2} w^2 \le \int_{ B_r^+} y^a |\nabla w|^2$.

(2) $C_a \int_{(\partial B_r)^+} y^{a-1} w^2 \le \int_{B_r^+} y^a|\nabla w|^2$.

(3) $C_a \int_{(\partial B_r)^+} y^{a} w^2 \le r\int_{B_r^+} y^a|\nabla w|^2$.

In these inequalities, the constants $C_a$ depend only on the value of $a$. 

\end{proposition}

\section{Some Liouville-type theorems}

In this section, we recall  Liouville-type theorems  for weak solutions of certain degenerate elliptic PDEs. Additionally, we prove a Liouviile-type theorem for subharmonic functions defined in the upper half-space.

\subsection{Liouville-type theorems for degenerate elliptic PDEs}
This subsection focuses on Liouville-type theorems for weak solutions of a specific class of degenerate PDEs as demonstrated in \cite{STV2}. The particular equation under consideration is:
\begin{equation} \label{eq:delta_a}
 \Delta_a u = 0\, \text{ on } \pR
\end{equation}

\begin{definition}
Let $a \in (-\infty, 1)$. A function $w \in \tilde{H}^{1, a}(\pR)$ is a weak solution of the (\ref{eq:delta_a}) if
$$\int_{\pR} y^a \nabla w \cdot \nabla \phi = 0 $$
for every $\phi \in C_c(\pR)$.
\end{definition}
Note that $\tilde{H}^{1, a}(\pR)$ is the natural functional space for weak solutions of the degenerate elliptic PDE (\ref{eq:delta_a}). Particularly when $a$ is within $(-1, 1)$, this equation is used in the celebrated work by Caffarelli and Silvestre \cite{CC} to characterize the fractional Laplacian as the Dirichlet-to-Neumann map.

The following Liouville-type theorems have been established in the work by Sire, Terracini, and Vita \cite{STV2}.

\begin{theorem}\label{thm:stv2} \cite[Theorem 3.1]{STV2}
Let $a \in (-1, 1)$, and $w$ be a solution to 
\begin{equation*}
\begin{cases}
\Delta_a w = 0 \text{ in } \pR \\
w = 0 \text{ on } \oR,
\end{cases}
\end{equation*}
and let us suppose that for some $d \in [0, 1-a)$, $C>0$ it holds
$$|w(z)| \le C(1+|z|^d)$$
for every $z$. Then $w$ is identically zero.
\end{theorem}

\begin{theorem}\label{stv2} \cite[Theorem 3.2]{STV2}
Let $a \in (-\infty, 1)$, and $w$ be a solution to 
\begin{equation*}
\begin{cases}
\Delta_a w = 0 \text{ in } \pR \\
w = 0 \text{ on } \oR,
\end{cases}
\end{equation*}
and let us suppose that for some $d \in [0, 1)$, $C>0$ it holds
$$|w(z)| \le C y^{2-a}(1+|z|^d)$$
for every $z$. Then $w$ is identically zero.
\end{theorem}

It's also noteworthy that for $m<\gamma<m+1$, where $m$ is a positive integer, Yang \cite{Y} exploits the fact that if $w$ is a solution to the equation (\ref{eq:delta_a}), then $\Delta_b^{m+1} w = 0 $ for $b= 2m+1-2\gamma \in (-1, 1)$. Consequently, Yang assumes that $w$ is an element of $H^{k,b}(\pR^{n+1})$ and successfully extends the results of \cite{CC} for general higher-order fractional Laplace operators.


\subsection{Liouville-type theorem for subharmonic functions}

In this subsection, we establish a Liouville-type theorem for subharmonic functions, which will be fundamental in proving the main theorems of this article. We begin by revisiting a sharp Poincar\'{e}-type inequality from \cite{STV1}.

\begin{lemma}\cite[Lemma 3.2]{STV1}
Let $w \in \tilde{H}^{1,0}(\pR)$. Then, the following inequality holds:
$$r\int_{B_r^+} |\nabla w|^2 \ge \int_{(\partial B_r)^+} w^2.$$
\end{lemma}

\begin{proof}
The proof is included for clarity. Suppose $w \in C_c^{1}(\overline{B_1^+}\setminus \oR)$. The subsequent calculation leads to the desired inequality:
\begin{align*}
\int_{(\partial B_1)^+} w^2 &= \int_{(\partial B_1)^+} \frac{\partial y}{\partial \nu} \frac{w^2}{y} = \int_{B_1^+} \nabla y \cdot \nabla \big(\frac{w^2}{y} \big) \\
& = \int_{B_1^+} |\nabla w|^2 - \big|\nabla w - \frac{w}{y}\nabla y \big|^2 \le \int_{B_1^+} |\nabla w|^2.
\end{align*}
\end{proof}

We are now ready to present the proof of a Liouville-type theorem applicable to non-negative subharmonic functions defined in the upper half-plane.
\begin{theorem}\label{thm:liouville}
Consider $w \in \tilde{H}^{1,0}(\pR)$ that satisfies the condition $\Delta w \ge 0$ in the upper half-plane in the following weak sense:
\begin{equation}\label{eq:sub}
\int_{\pR} \nabla \eta \cdot  \nabla  w \le 0
\end{equation}
for any non-negative $\eta \in C_c^1(\pR)$. Additionally, assume that $0\le w \le C(y^d+1)$ for some positive constant $C$ and for $d \in [0,1)$.
Then, $w$ is identically zero.
\end{theorem}

\begin{proof}

Let $H(r) = \frac{1}{r^{n}} \int_{(\partial B_r)^+ }  w^2 $ and $E(r) = \frac{1}{r^{n-1}} \int_{ B_r^+ }  |\nabla w|^2$. We test the inequality (\ref{eq:sub}) with $\eta = \eta_{r, \epsilon} w$ where $\eta_{r, \epsilon}$ is a cut-off function defined by:
$$
\eta_{r, \epsilon}(z) = \begin{cases} 1  \text{ for }  |z|\le r\\
1- \frac{|z|-r}{\epsilon}  \text{ for }  r \le |z|\le r+\epsilon \\
0 \text{ for } |z|\ge r+\epsilon.
\end{cases}$$
Consequently, we obtain
\begin{equation*}
\int_{\pR} \eta_{r, \epsilon} |\nabla w|^2 \le -\int_{\pR} w \nabla \eta_{r, \epsilon} \cdot \nabla w = \frac{1}{2\epsilon} \int_{B_{r+\epsilon}^+ \setminus B_r^+} \partial_r w^2 
\end{equation*}
As $\epsilon \rightarrow 0$, the left-hand side converges to $\int_{B_r^+} |\nabla w|^2$, and the right-hand side converges to $\frac{1}{2} \int_{(\partial B_r)^+} \partial_r w^2$ for almost every $r$.
Therefore, we establish the inequality
$$E(r) = \frac{1}{r^{n-1}} \int_{ B_r^+ }  |\nabla w|^2 \le \frac{1}{2r^{n-1}}  \int_{(\partial B_r)^+} \partial_r w^2 = \frac{r}{2} H'(r).$$
for almost every $r$. The Poincar\'{e} inequality then implies:
$$H(r) \le E(r) \le \frac{r}{2} H'(r).$$
This leads to the differential inequality $\big(\frac{1}{r^2}H(r)\big)' \ge 0$. Integrating this inequality, we find that for any fixed $r_0$ and for all $r \ge r_0$, 

$$\frac{1}{r_0^2} H(r_0) \le \frac{1}{r^2} H(r).$$
Given the bound $0\le w \le C(y^d+1)$, it is easy to see that $H(r) \le C (r^{2d}+1)$.
Consequently, we deduce
$$ \frac{1}{r_0^2} H(r_0) \le \frac{1}{r^2} H(r) \le C (r^{2d-2} + r^{-2}).$$
Since $d \in [0, 1)$, letting $r$ tend to infinity leads to the conclusion that $H(r_0) = 0$. This implies $w = 0$ everywhere.

\end{proof}

\begin{remark}
The condition $d<1$ is necessary because  the function $y$ is a non-negative harmonic function vanishing on $\oR$.
\end{remark}

\begin{remark}
The same technique can be applied even when the assumption $w \ge 0$ is not present, provided that  $w$ solves a PDE of the form $\Delta w = F(w)w$, with $F$ being non-negative and sufficiently regular.
\end{remark}

\section{Conformal infinity with positive scalar curvature}
In this section, we provide a detailed proof for Theorem 1.1.
\subsection{Previously known properties}
We start by summarizing some established geometric properties relevant to CCE manifolds under the condition of positive scalar curvature at the conformal infinity.

For $1\le \gamma \le \frac{n}{2}+1$, the positivity of $R_h$ implies the positivity of scalar curvature $R_s^*$.
\begin{proposition} \cite[Proposition 6.4]{CC} \label{lemma:sign2}
Let $(X, M , g_+)$ be a CCE manifold with conformal infinity $(M, [h])$. Suppose the scalar curvature $R_h$ of $(M, h)$ is positive. Then, the scalar curvature $R_s^*$ of the adapted metric $g_s^*$ is positive when $1\le \gamma \le \frac{n}{2}+1$.
\end{proposition}

This proposition establishes a direct correlation between the positivity of scalar curvature in the boundary geometry and the scalar curvature of the adapted metric within a specific range of $\gamma$.

\begin{theorem}\cite[Theorem 1.2]{GQ} \label{thm:asym}
Assume $0<\gamma<1$. The fractional curvature $Q_{2\gamma} \coloneqq \frac{2}{n-2\gamma} P_{2\gamma}(1)$ is positive, provided that the scalar curvature $R_h$ of $(M, h)$ is positive.
\end{theorem}
Given that $P_{2\gamma} \coloneqq 2^{2\gamma} \frac{\Gamma(\gamma)}{\Gamma(-\gamma)} S(\frac{n}{2} + \gamma )$ and noting that $\Gamma(-\gamma)$ is negative while $\Gamma(\gamma)$ is positive, it follows that the asymptotic formula for $v_s$ is $y^{n-s}(1+ \alpha_{s}y^{2\gamma} + \cdots)$ where $\alpha_s$ is a negative function defined on the boundary. This observation  will be utilized later in the proof of Theorem 1.1.

Finally, in the case $\gamma = \frac{1}{2}$, the scalar curvature $R_s^*$ is identically zero. Instead, the positivity of $R_h$ implies the following:
\begin{lemma}\cite[Lemma 5.4]{WZ} \label{lemma:wz}
If $R_h>0$, then for $\gamma = \frac{1}{2}$, $\big( \frac{1-|\nablas \rho_s|^2}{\rho_s} \big) >0$ on $X$, and $R_s^* \equiv 0$.
\end{lemma}

\subsection{Proof for Theorem \ref{mainthm2.1}}

First we recall some conformal transformation rules for curvature quantities.

\begin{lemma} \label{lemma:formulas}
Let $\frac{n}{2}<s<n$ and $f_s \coloneqq 1- |\nablas \rho_s|^2$. Denote covariant derivative, scalar curvature, traceless Ricci curvature, and Ricci curvature of $g_s^*$ by $\nablas, R_s^*, E_s^*, Ric_s^*$ respectively. Then we have following identities:

(1) $\Delta_s^* \rho_s  = -s(\frac{f_s}{\rho_s}).$

(2) $R_s^*  =2n (s- \frac{n+1}{2}) \frac{f_s}{\rho_s^2}.$

(3) $E_s^*  = -(n-1)\rho_s^{-1} ( \nablas^2 \rho_s + \frac{s}{n+1} \frac{f_s}{\rho_s} g_s^*).$

(4) $Ric_s^*  = -(n-1) \rho_s^{-1} (\nablas^2 \rho_s + \frac{s}{n+1} \frac{f_s}{\rho_s} g_s^* ) +(\frac{s-\frac{n+1}{2}}{n+1}) \frac{f_s}{\rho_s^2} g_s^*.$

(5) $ Ric_s^*(\nablas \rho_s , \nablas, \rho_s ) = -(n-1)\frac{\nablas^2 \rho_s (\nablas \rho_s, \nablas \rho_s)}{\rho_s} + [\frac{s-\frac{n+1}{2}}{n+1} - \frac{(n-1)s}{n+1}] \frac{f_s}{\rho_s^2} |\nablas \rho_s|^2.$

\end{lemma}
\begin{proof}
For the detailed derivations and proofs of these formulas, refer to equations (2.2), (2.3), (4.3), (4.4), and (4.5) of \cite{WZ}.
\end{proof}

\begin{theorem} \label{lemma:sign}
Let $(X^{n+1}, \partial X = M , g_+)$ be a CCE manifold with conformal infinity $(M, [h])$. Assume that the scalar curvature $R_h$ of $(M, h)$ is positive. For $0<\gamma<1$, it follows that $|\nablas \rho_s|^2 < 1$ on $\mathring{X}$, or equivalently, $\mathrm{sgn}(R_s^*) = \mathrm{sgn}(\gamma - \frac{1}{2})$.
\end{theorem}

\begin{proof}
A straightforward calculation using Lemma \ref{lemma:formulas} gives the following equation for $f_s$:
\begin{equation} \label{eq:fs}
\Delta_s^* f_s = - 2|\nablas^2 \rho_s |^2 + (2s-n+1) \frac{\nablas \rho_s}{\rho} \cdot \nablas f_s - 2(2s-n) f_s\frac{|\nablas \rho_s |^2}{\rho_s^2}.
\end{equation}

Next, we compute the asymptotic expansion of $f_s$. Let $y$ be the geodesic defining function, denote $g \coloneqq y^2 g_+$ and $\nabla$ be the covariant derivative with respect to $g$.
According to Theorem \ref{thm:asym}, we have
\begin{equation*}
v_s = y^{n-s}(1+ \alpha_{s}y^{2\gamma} + \cdots)
\end{equation*}
for some negative function $\alpha_s$ defined on $M$. Direct computation leads to
\begin{equation*}
\rho_s = y(1+\frac{\alpha_s}{n-s} y^{2\gamma} +\cdots )    
\end{equation*}

which, in turn, gives us
\begin{align*}
f_s = 1-|\nablas \rho_s|^2 = & 1- (\frac{y}{\rho_s})^2 |\nabla \rho_s |^2 \\
= & 1 - (1 + \frac{\alpha_s}{n-s} y^{2\gamma} )^{-2} (1 + (1+2\gamma) \frac{\alpha_s}{n-s} y^{2\gamma} ) ^2 \\
= & - 4\gamma \frac{\alpha_s}{n-s}  y^{2\gamma} + \cdots
\end{align*}
Given that $\alpha_s < 0$, the expression  $\frac{f_s}{\rho_s^{2\gamma}}$ is strictly positive on the boundary $M$. This is the underlying  reason for choosing the exponent $2\gamma$ in the denominator of the quantity $\frac{f_s}{\rho_s^{2\gamma}}$.

We now compute the equation for $\frac{f_s}{\rho_s^{2\gamma}}$. Using the identity (\ref{eq:fs}), it is easy to check
\begin{equation}\label{eq:mp}
\Delta_s^*(\frac{f_s}{\rho_s^{2\gamma}} ) =  - \frac{2 |\nablas^2 \rho_s | ^2 }{\rho_s^{2\gamma}} +(1-2\gamma) \nablas (\frac{f_s}{\rho_s^{2\gamma}})\cdot (\frac{\nablas \rho_s}{\rho_s}) + 2\gamma s \frac{f_s^2}{\rho_s^{2\gamma+2}}.
\end{equation}

To prove that $\frac{f_s}{\rho_s^{2\gamma}}$ is positive, we employ the continuity method. For the case $s = \frac{n}{2} + \frac{1}{2}$, or equivalently $\gamma = \frac{1}{2}$, Lemma \ref{lemma:wz} gives the positivity.  We extend this to all other values of $\gamma \in (0,1)$ 

 Since $v_s$ is a continuous family with respect to $s$, the set of $s$ values for which $\frac{f_s}{\rho_s^{2\gamma}}>0$ is open. To prove closedness, assume $\frac{f_s}{\rho_s^{2\gamma}}\ge 0$ and $\frac{f_s}{\rho_s^{2\gamma}}$ attains 0 at some point in $X$. Given the strict positivity of $\frac{f_s}{\rho_s^{2\gamma}}$ on the boundary(as derived from its asymptotic expansion), any zero of $\frac{f_s}{\rho_s^{2\gamma}}$ must occur at an interior point of $X$.  We apply Serrin's strong maximum principle(see \cite[Theorem 2.10]{HL}) to the equation (\ref{eq:mp}) in the region $X\setminus \{y<\epsilon\}$ for a sufficiently small $\epsilon$. This application leads to a contradcition, implying that $f_s>0$ throughout $X$.

Consequently, the sign of $R_s^*$ follows from the identity $R_s^*  = 2n(\gamma -\frac{1}{2}) \frac{f_s}{\rho_s^2}$.
\end{proof}

\begin{corollary}
Under the same assumption to that of Theorem \ref{lemma:sign}, we have $\lim_{y \rightarrow 0} \rho_s^{2-2\gamma} R_s^* = C_{n,s} Q_{2\gamma}$ and $|R_s^*|\ge C \rho_s^{2\gamma-2}$ for some positive constant $C$. For $0<\gamma<1/2$, the constant $C$ can be chosen as $C = C_{n,s} \inf_M Q_{2\gamma}$ where $C_{n,s}$ is a constant depending only on $n$ and $s$.
\end{corollary}
\begin{proof}
The first statement is derived from the identity: 
$$R_s^*  = 2n (s- \frac{n+1}{2}) \frac{f_s}{\rho_s^2} = C_{n,s} \big( \frac{f_s}{\rho_s^{2\gamma}} \big) \cdot \rho_s^{2\gamma-2}.$$
For $0<\gamma <1/2$, using the identity (\ref{eq:mp}), we obtain:
\begin{align*}
\Delta_s^*(\frac{f_s}{\rho_s^{2\gamma}} )& =  - \frac{2 |\nablas^2 \rho_s | ^2 }{\rho_s^{2\gamma}} +(1-2\gamma) \nablas (\frac{f_s}{\rho_s^{2\gamma}})\cdot (\frac{\nablas \rho_s}{\rho_s}) + 2\gamma s \frac{f_s^2}{\rho_s^{2\gamma+2}} \\
& \le  - \big(\frac{2}{n+1} \big)\cdot \frac{(\Delta_s^* \rho_s)^2 }{\rho_s^{2\gamma}} +(1-2\gamma) \nablas (\frac{f_s}{\rho_s^{2\gamma}})\cdot (\frac{\nablas \rho_s}{\rho_s}) + 2\gamma s \frac{f_s^2}{\rho_s^{2\gamma+2}} \\
& = (1-2\gamma) \nablas (\frac{f_s}{\rho_s^{2\gamma}})\cdot (\frac{\nablas \rho_s}{\rho_s}) + \big(\frac{2sn(\gamma-1/2)}{n+1} \big)\cdot \frac{f_s^2}{\rho_s^{2\gamma+2}} \\
&\le  (1-2\gamma) \nablas (\frac{f_s}{\rho_s^{2\gamma}})\cdot (\frac{\nablas \rho_s}{\rho_s}).
\end{align*}
The application of the maximum principle leads to the second statement.
\end{proof}

\begin{remark}
The identity $\lim_{y \rightarrow 0} \rho_s^{2-2\gamma} R_s^* = C_{n,s} Q_{2\gamma}$ suggests that $\rho_s^{2-2\gamma} |Rm_s^*|$ is an appropriate quantity to control near the boundary as $|Rm_s^*|$ blows up as $y \rightarrow 0$.
\end{remark}

\section{Liouville-type theorems on the hyperbolic space}

In this section, we establish Liouville-type theorems on the standard upper half-plane model of hyperbolic space.
\subsection{Liouvile-type theorems for adapted metrics with $0<\gamma<\frac{n}{2}$}

Throughout this subsection, we consider positive functions $v_s$ and denote $\rho_s = v_s^{1/(n-s)}$ for $\frac{n}{2} < s < n$.  Additionally, we define the metric $g_s^* \coloneqq \frac{\rho_s^2}{y^2} |dz|^2$ on $\pR$.
Our starting point is a lemma concerning the condition $1-|\nabla^*_s \rho_s|^2 \ge 0$.
\begin{lemma}\label{lemma:ydirection}
Let $\frac{n}{2} < s < n$ and $v_s$ be a positive function defined on $\pR$ satisfying $1-|\nabla^*_s \rho_s|^2 \ge 0$ . Define $w_s \coloneqq 1-\frac{v_s}{y^{n-s}}$ and assume $w\equiv 0$ on $\oR$. Then $\partial_y w_s$ is non-negative, and $0 \le w_s<1$ on $\pR$.
\end{lemma}
\begin{proof}
Observe that $1-|\nabla^*_s \rho_s|^2 \ge 0$ implies $|\nablas v_s^{1/(n-s)}| \le 1$, which further implies $\frac{y}{v_s^{1/(n-s)}} |\nabla v_s^{1/(n-s)}| \le 1$ where $\nabla$ is the convariant derivative with  respect to the Euclidean metric.
 In particular, we have
$$\frac{v_s^{1/(n-s)}}{y} \ge  |\nabla v_s^{1/(n-s)}| \ge \partial_y v_s^{1/(n-s)}=  v_s^{1/(n-s) -1 } (\frac{1}{n-s}) \partial_y v_s,$$
which implies that $(n-s)\frac{v_s}{y} \ge \partial_y v_s$ since $s<n$.

Therefore, we deduce $$\partial_y w_s = \partial_y \big( 1-\frac{v_s}{y^{n-s}} \big) =\frac{1}{y^{n-s}}((n-s)\frac{v_s}{y} - \partial_y v_s ) \ge 0 $$.
$w_s<1$ follows from the fact that$v_s$ is positive. Furthermore, since  $w_s \equiv 0$ on $\oR$ and $w_s$ is non-decreasing in the $y$-direction, $w_s$ must be non-negative.
\end{proof}

Now we are ready to prove Liouville-type theorems.
\begin{theorem} \label{thm:2}
Let $\gamma \in (0,1)$. Suppose $v_s$ is a positive solution to the PDE $-\Delta_+ v_s -s(n-s)v_s = 0$ on $\pR$.   Define $w_s \coloneqq 1-\frac{v_s}{y^{n-s}}$. Assume the following conditions:

(1)  $w_s \in \tilde{H}^{1,a}(\pR)$.

(2) $|\nablas \rho_s | \le 1$. This is equivalent to $\mathrm{sgn}(R^*_s) = \mathrm{sgn}(\gamma - \frac{1}{2})$.

Then, $v_s = y^{n-s}$ or equivalently, $g_s^* \equiv |dz|^2$.
\end{theorem}
\begin{proof}
Firstly, we compute the PDE for $w_s$ with respect to the Euclidean metric:
$$\Delta_a w_s = 0.$$
Given $w_s \in \tilde{H}^{1,a}(\pR)$, we infer that $w_s \equiv 0$ on $\oR$ by Proposition \ref{lemma:trace}. Now Lemma \ref{lemma:ydirection} assures us that $0 \le w_s < 1$. 

As $w_s \in \tilde{H}^{1, a}(\pR)$, and satisfies $0 \le w_s \le 1$, we can apply Theorem \ref{thm:stv2} to conclude that $w_s \equiv 0$, which implies $v_s = y^{n-s}$.
\end{proof}

\begin{theorem} \label{thm:1}
Let $s \coloneqq \frac{n}{2} + \gamma$ for some $\gamma \in [\frac{1}{2}, \frac{n}{2})$. Suppose $v_s$ is a positive solution to the PDE $-\Delta_+ v_s -s(n-s)v_s = 0$  on $\pR$. Define $w_s \coloneqq 1-\frac{v_s}{y^{n-s}}$. Assume the following conditions:

(1) $w_s \in \tilde{H}^{1,0}(\pR)$.

(2) $R^*_s $ is non-negative.

Then, $v_s = y^{n-s}$ or equivalently, $g_s^* \equiv |dz|^2$.
\end{theorem}

\begin{proof}
A straightforward computation shows that the PDE for $w_s$ with respect to the Euclidean metric is 
$$\Delta w_s = \frac{(2\gamma-1)}{y}\partial_y w_s.$$
Firstly, we observe from  Lemma \ref{lemma:ydirection} that $ 0 \le \partial_y w_s$, and $0 \le w_s \le 1$. Additionally, the linear PDE above indicates that $w_s \in C^{\infty}_{\mathrm{loc}}(\pR)$, implying that $\Delta w_s \ge 0$ at every point on $\pR$.

By applying Theorem \ref{thm:liouville}, we deduce that $w_s\equiv 0$.
\end{proof}

\begin{remark}
In the absence of constraints on the scalar curvature, alternative solutions to the PDE for $v_s$ exist. Specifically, $v_s = y^{n-s} + \beta y^s$ can also be a solution for a positive real number $\beta$.
\end{remark}

\begin{remark}
For Theorem \ref{thm:1}, it is essential to note that the assumption is $w_s \in \tilde{H}^{1, 0}(\pR)$ rather than $\tilde{H}^{1, a}(\pR)$.  While $\tilde{H}^{1, a}(\pR)$ might seem a more intuitive functional space for weak solutions,  the space $\tilde{H}^{1, 0}(\pR)$ is still sufficient to prove the Liouville-type theorem. Moreoever, according to Proposition \ref{lemma:inclusion}, $\tilde{H}^{1, a}(\pR) \subset \tilde{H}^{1, 0}(\pR)$ when $\gamma \ge \frac{1}{2}$. Yang's work \cite{Y} requires integrability of higher-order derivatives when $\gamma>1$; under suitable boundary assumptions, this implies that the solution is in $\tilde{H}^{1, 0}(\pR)$ by Proposition \ref{lemma:whardy}. Therefore, the regularity assumption on $w_s$ in our context appears to be minimal.
\end{remark}

\begin{remark}
Theorem \ref{stv2} specifies the behavior of $w$ near the boundary by imposing $|w(z)| \le C y^{2-a}(1+|z|^d)$, while Theorem \ref{thm:1} only assumes boundedness.
\end{remark}

\subsection{Liouville-type theorem for compactified metrics}

The proof of Theorem \ref{thm:1} reveals an interesting aspect: the necessity of an explicit PDE condition for establishing a Liouville-type theorem is not absolute. This insight allows us to formulate a version of the theorem that relies more on geometric and boundary conditions than on a specific PDE. 
\begin{theorem} \label{thm:3}
Let $\rho$ be a positive function defined on $(\pR, g_+ \coloneqq \frac{1}{y^2}|dz|^2)$ with following properties:

(1) $\rho \in C^2(\overline{\mathbb{R}}_+^{n+1})$

(2) $\rho(x, 0) = 0 $ and $\partial_y \rho(x, 0) = \lim_{y\rightarrow 0 }\frac{\rho}{y}(x, y) = 1$, or equivalently, the induced boundary metric on $\oR$ of the compactified manifold $(\overline{\mathbb{R}}_+^{n+1}, g^* \coloneqq \rho^2 g_+)$ is the Euclidean metric.

(3) The scalar curvature of the compactified metric, $R^*$, is non-negative.

(4) $|\nabla^* \rho| \le 1$.

Then, $\rho \equiv y$, or equivalently, $g^* \equiv |dz|^2$.

\end{theorem}

\begin{proof}
We begin by considering the implication of the inequality $|\nabla^* \rho| \le 1$. This results in the following inequalities:
$$\partial_y \rho \le |\nabla \rho| = \frac{\rho}{y} |\nabla^* \rho| \le \frac{\rho}{y}.$$
Integrating this differential inequality gives us that for any $(x, y) \in \pR$ and $\epsilon>0$, we have:
$$\log {\rho (x, y)} - \log {\rho (x, \epsilon)} \le \log{y} - \log{\epsilon}.$$
This implies $\log {(\rho(x, y)/y)} \le \lim_{\epsilon \rightarrow 0} \log {(\rho(x, \epsilon)/\epsilon)} = \log {1} = 0$, or in other words, $\rho(x, y) \le y$.

The non-negativity of the scalar curvature $R^*$ is equivalent to the following condition dervied from the direct computation:
$$\Delta \log(\frac{\rho}{y}) + \frac{(n-1)}{2} |\nabla \log(\frac{\rho}{y})|^2 \le 0.$$

Define $w(x, y) \coloneqq \big( \frac{\rho(x, y)}{y} \big)^{\frac{(n-1)}{2}}$, which lies in $C^1(\overline{\pR}) \cap C^2_{\mathrm{loc}} (\pR)$.  The inequality for $\log(\frac{\rho}{y})$ is equivalent to $\Delta w \le 0$, and the inequality $\rho(x, y) \le y$ is equivalent to $0<w \le 1$. Moreover, $w(x, 0) = 1$.

Applying Theorem \ref{thm:liouville} to the function $1-w$ , we deduce that $w \equiv 1$. This completes the proof.
\end{proof}

\begin{remark}
The non-negativity of $R^*$ by itself is not a sufficient. A counter-example is provided by the function $\rho = y (1 + \log(1+y))^{\frac{2}{n-1}}$. For this specific choice of $\rho$, the function $w(x, y) = 1 + \log(1+y)$ used in the proof of the theorem satisfies $\Delta w \le 0$ and $w(x, 0) = 1$, yet it does not lead to the conclusion that $\rho \equiv y$.
\end{remark}

\begin{remark}
If $\rho$ is a polynomial in $y$ with degree $d$,  then it is easy to check that $\lim_{y \rightarrow \infty} |\nabla^* \rho| = d$. This observation suggests that the assumption $|\nabla^* \rho| \le 1$ serves as an intrinsic way (with respect to the metric $g^*$) to control the growth of $\rho$ as $y \rightarrow \infty$.
\end{remark}

A direct outcome of the Theorem \ref{thm:3} and its proof is a Liouville-type theorem for adapted metrics with $\gamma \in [\frac{n}{2}, \frac{n}{2}+1]$ including the Fefferman-Graham compactification.
\begin{corollary} 

Suppose $v$ is a positive solution to the PDE $-\Delta_+ v_n -n = 0$  on $\pR$. Define the metric $g_n^* \coloneqq \frac{e^{2v_n}}{y^2}|dz|^2$ and $w_n \coloneqq 1 - \frac{e^{v_n}}{y}$. Assume the following conditions:

(1) $w_n \in \tilde{H}^{1,0}(\pR)$.

(2) $R^*_n $ is non-negative.

Then, $w_n \equiv 0$ or equivalently, $g_n^* \equiv |dz|^2$.
\end{corollary}

\begin{proof}
Let us define $\rho_n = e^{v_n}$. The conformal transformation rule for the scalar curvature implies
$$0\le R^*_n=  \frac{1}{\rho_n^2}(-n(n+1) -2n\Delta_+  v_n-n(n-1) |\nabla_+ v_n|^2) = \frac{n(n-1)}{\rho_n^2}(1-|\nabla_+ v_n|^2).$$
From this, it follows that $|\nabla^*_n \rho_n |^2 = \rho_n^2 |\nabla^*_n v_n|^2 =  |\nabla_+ v_n|^2 \le 1$. Additionally, note that  $0 = w_n(x, 0) = \lim_{y \rightarrow 0} (1-\frac{\rho_n(x, y)}{y})$, which is equivalent to $\lim_{y \rightarrow 0} \frac{\rho_n(x, y)}{y} = 1$.

Furthermore, the non-negativity of $R^*_n $ leads to the inequality $\Delta w_n \ge 0$. Thus, we can follow the same approach as in Theorem \ref{thm:3} to deduce $w_n \equiv 0$.

\end{proof}

\begin{corollary} 
Let $s \coloneqq \frac{n}{2} + \gamma$ for some $\gamma \in (\frac{n}{2}, \frac{n}{2}+1]$. Suppose $v_s$ is a positive solution to the PDE $-\Delta_+ v_s -s(n-s)v_s = 0$  on $\pR$. Define $w_s \coloneqq 1-\frac{v_s}{y^{n-s}}$. Assume the following conditions:

(1) $w_s \in \tilde{H}^{1,0}(\pR)$.

(2) $R^*_s $ is non-negative.

Then, $v_s = y^{n-s}$ or equivalently, $g_s^* \equiv |dz|^2$.
\end{corollary}

\begin{proof}
The identity $R_s^*  = 2n(s- \frac{n+1}{2}) \cdot \frac{(1-|\nablas \rho_s|^2)}{\rho_s^2}$ implies that $|\nablas \rho_s|\le 1$. The remainder of the proof follows similarly to the previous corollary.
\end{proof}

\section{Cacciopoli-type inequalities}

In this section, we establish an upper bound on the $y^a$-weighted Sobolev norm for solutions $w_s$ that appear in Theorem \ref{thm:2}, specifically in a CCE setting. To achieve this, we will prove some Cacciopoli-type inequalities. Conceptually, such upper bound is essential for demonstrating the existence of a limit function in the weighted Sobolev space, especially for a sequence of functions $(w_s)_i$ defined on a sequence of CCE manifolds $(X_i, M_i, (g_+)_i)$ during a blow-up analysis. We articulate this bound in terms of geometric quantities:  namely, the $y^a$-weighted volume of geodesic balls and the $L^1$-norm of the  fractional curvature $Q_{2\gamma}$ along the boundary.

Recall that on a CCE manifold $(X, M, g_+)$, the metric can be expressed as $g_+ = \frac{1}{y^2}(dy^2 + g_y)$ where $y$ is the geodesic defining function and $g_0 = h$. Note that $y$ was also used as a geodesic defining function in the hyperbolic space $\pR$.  The function $y$ is defined on a collar neighborhood $\{y<\epsilon\} \simeq M\times [0, \epsilon)$ for some positive number $\epsilon$. 

In this section, we continue to work under the assumption that the scalar curvature $R_h$ of the manifold $(M, h)$ is positive. We define $g \coloneqq y^2 g_+$ and denote by $\nabla$, $R$ and $Ric$ the covariant derivative, scalar curvature, and Ricci curvature, respectively, with respect to the metric $g$.

Based on the Theorem \ref{lemma:sign}, we know that $|\nabla_s \rho_s| \le 1$. Let's define $w_s = 1 - \frac{v_s}{y^{n-s}}$, as before.  Let's prove some basic observations about the region $M\times [0, \epsilon)$.

\begin{lemma}\label{lemma:ydirection1}
For $y<\epsilon$, the derivative $\partial_y w_s$ is non-negative, and $0\le w_s (x, y) \le 1$.
\end{lemma}

\begin{proof}
By a direct computation, we observe that the inequality $|\nabla_s \rho_s| \le 1$ is equivalent to 
$$|\nabla w_s|^2 \le 2(n-s)(1-w_s) \frac{\partial_y w_s}{y}.$$
Given that $|\nabla y|^2 = 1$, it follows that $|\nabla w_s|^2 \ge |\nabla w_s \cdot \nabla y |^2 = (\partial_y w_s)^2$. This results in the inequality $\partial _y w_s \ge 0$. Additionally, since $w_s \equiv 0$ on $M$, and $v_s > 0$, it is clear that $0\le w_s (x, y) \le 1$ for $y<\epsilon$. 
\end{proof}

\begin{lemma}\label{lemma:Rpositive}
The scalar curvature $R$ of the metric $g$ is positive in the collar neighborhood $M\times[0, \epsilon)$.
\end{lemma}

\begin{proof}
Recall some identities from \cite[Lemma 5.2]{L}: $\Delta y = -\frac{1}{2n}yR$, $Ric(\nabla y, \nabla y) = \frac{1}{2n} R$ in the collar neighborhood $M\times[0, \epsilon)$, and $R = \frac{n}{n-1} R_h>0$ on the boundary $M$.

Applying the Laplacian on the equation $|\nabla y |^2 =1 $ and using Bochner's formula, we obtain:
\begin{align*}
0 = \frac{1}{2}\Delta(|\nabla y|^2) &= |\nabla^2 y|^2 + \nabla y \cdot \nabla(\Delta y) + Ric(\nabla y, \nabla y)\\
& = |\nabla^2 y|^2 + \partial_y (-\frac{1}{2n}yR) + \frac{1}{2n} R \\
& \ge -\frac{1}{2n} y \partial_y R
\end{align*}
From this inequality, we deduce that $\partial_y R \ge 0$. Given the scalar curvature $R_h$ of the boundary metric is positive, it follows that the scalar curvature $R$ within the collar neighborhood $M\times[0, \epsilon)$ is also positive.
\end{proof}

The PDE for $w_s$ with respect to the metric $g$ is given as follows:

\begin{equation}\label{eq}
\mathrm{div}(y^a\nabla w_s) = (n-s)\frac{(1-w_s)}{y^{2\gamma}} \Delta y
\end{equation}

Using this elliptic equation, we prove an upper bound on the weighted Sobolev norms of $w_s$ within half-balls, as required by the assumptions of Theorem \ref{thm:2}. Here, $B_r^+$ represents a geodesic ball with radius $r$, centered at a point on the boundary.

\begin{proposition} 
Let $0 < \gamma < 1$ and assume that the scalar curvature of $(M, h)$ is positive. Then, for $r < \frac{\epsilon}{2}$, we have
$$\int_{B_r^+}y^a\big( w_s^2 + |\nabla w_s|^2 \big)dV_g\le C(n, \gamma)\big((1+\frac{1}{r^2})\int_{B_{2r}^+}  y^a dV_g + \oint_{B_{2r}^+\cap M} Q_{2\gamma} dV_h \big).$$
where $C(n, \gamma)$ is a constant depending on $n$ and $\gamma$.
\end{proposition}

\begin{proof}
Let $u_s = 1-w_s$. It is straightforward to see from Lemma \ref{lemma:ydirection1} that $0 \le u_s \le 1$ and $u_s \equiv 1$ on $M$. The equation (\ref{eq}) and Lemma \ref{lemma:Rpositive} imply that
$$\mathrm{div}(y^a\nabla u_s) = -(n-s)\frac{u_s}{y^{2\gamma}} \Delta y =  \frac{(n-s)}{2n}y^au_s R \ge 0.$$
We test the above elliptic differential inequality with a function $\eta^2 u_s$ where $\eta$ is a Lipshitz cut-off function satisfying $\eta \equiv 1$ on $B_r$,  $\eta \equiv 0$ outside $B_{2r}$, and $|\nabla \eta|\le C/r$ for a dimensional constant $C$. We have the following inequality:
\begin{align*}
\int_{B_{2r}^+\cap\{y\ge\epsilon_1\}} y^a \eta^2 |\nabla u_s|^2 &\le \int_{B_{2r}^+\cap\{y\ge\epsilon_1\}} -2y^a\eta u_s \nabla \eta \cdot \nabla u_s - \oint_{B_{2r}^+\cap\{y = \epsilon_1\}} y^a \eta^2 u_s \frac{\partial u_s}{\partial y}.
\end{align*}
Then, by Young's inequality, we have 
$$\int_{B_{2r}^+\cap\{y\ge\epsilon_1\}} y^a \eta^2 |\nabla u_s|^2  \le 4\int_{B_{2r}^+\cap\{y\ge\epsilon_1\}} y^a |\nabla \eta|^2 u_s^2 - 2 \oint_{B_{2r}^+\cap\{y = \epsilon_1\}} y^a \eta^2 u_s \frac{\partial u_s}{\partial y}. $$

As we take the limit $\epsilon_1 \rightarrow 0$, we calculate that the boundary integral term converges  to 
$$\lim_{\epsilon_1 \rightarrow 0}  -  \oint_{B_{2r}^+\cap\{y = \epsilon_1\}} y^a \eta^2 u_s \frac{\partial u_s}{\partial y} = \oint_{B_{2r}^+\cap M} \eta^2 C_{n, \gamma} Q_{2\gamma}$$
where $C_{n, \gamma}$ is a positive constant. This proves that
$$\int_{B_r^+}y^a|\nabla w_s|^2 dV_g\le C(n, \gamma)\big(\frac{1}{r^2}\int_{B_{2r}^+}  y^a dV_g + \oint_{B_{2r}^+\cap M} Q_{2\gamma} dV_h \big).$$

\end{proof}

For the case $\frac{1}{2}\le \gamma <1$, we provide another Cacciopoli-type inequality to give an upper bound on the $W^{1,2}$-norm of $w_s$ as required by the assumptions of Theorem \ref{thm:1}. Given the conformal transformation of the scalar curvature, we have the following inequality:
\begin{equation}\label{eq:scalarequation}
\frac{4n}{n-1} \Delta (\frac{\rho_s}{y})^{\frac{n-1}{2}} = (\frac{\rho_s}{y})^{\frac{n-1}{2}} R - (\frac{\rho_s }{y})^{\frac{n+3}{2}}R_s^*  \ge - (\frac{\rho_s }{y})^{\frac{n+3}{2}}R_s^*.
\end{equation}

\begin{proposition} 
Let $\frac{1}{2} < \gamma < 1$. Assume that the scalar curvature of $(M, h)$ is positive. Then, for $r < \frac{\epsilon}{2}$
$$\int_{B_r^+} \big(  w_s^2 + |\nabla w_s|^2 \big) dV_g\le C(n, \gamma)\int_{B_{2r}^+} \big( 1 + \frac{1}{r^2} +   y^{2\gamma-2} ||\rho_s^{2-2\gamma}R_s^*||_{L^\infty(X)} \big)dV_g$$
where $C(n, \gamma)$ is a constant depending on $n$ and $\gamma$.
\end{proposition}

\begin{proof}
From Lemma \ref{lemma:ydirection1}, we directly observe that $\rho_s \le y$ and thus $\int_{B_r^+}  w_s^2 \le  \int_{B_{r}^+}1 $. We now estimate $\int_{B_r^+} y^a|\nabla w_s|^2$. Let us denote $u_s \coloneqq \frac{\rho_s}{y}$. It follows that $\int_{B_r^+} |\nabla w_s|^2 = C\int_{B_r^+}  u_s^{2(n-s-1)}|\nabla u_s|^2$ and the inequality (\ref{eq:scalarequation}) is equivalent to $$\frac{4n}{n-1} \Delta u_s^{\frac{n-1}{2}} \ge - (u_s)^{\frac{n+3}{2}}R_s^*.$$

We test (\ref{eq:scalarequation}) with a function $\eta^2 u_s^{\frac{n+1}{2}-2\gamma}$, where $\eta$ is a Lipshitz cut-off function satisfying $\eta \equiv 1$ on $B_r$,  $\eta \equiv 0$ on $B_{2r}$, and $|\nabla \eta|\le C/r$ for a dimensional constant $C$. 
Then we have following computation:
\begin{align*}
\int_{B_{2r}^+} \eta^2  u_s^{\frac{n+1}{2}-2\gamma} \Delta u_s^{\frac{n-1}{2}} &= -C\int_{B_{2r}^+} \eta \nabla \eta u_s^{n-1-2\gamma} \nabla u_s - C \int_{B_{2r}^+}   \eta^2 |\nabla w_s|^2  \\
& \le -C(n,\gamma) \int_{B_{2r}^+} \big(   \eta^2 |\nabla w_s|^2 + C(n, \gamma) |\nabla \eta|^2 u_s^{n-2\gamma} \big),
\end{align*}
where the inequality follows from Young's inequality. Consequently, we obtain:
\begin{align*}
\int_{B_{2r}^+}   \eta^2 |\nabla w_s|^2 &\le C(n, \gamma)\int_{B_{2r}^+} \big( |\nabla \eta|^2 u_s^{n-2\gamma}  + u_s^{n+2-2\gamma}R_s^* \big) \\
& \le C(n, \gamma)\int_{B_{2r}^+}  \big( |\nabla \eta|^2 u_s^{n-2\gamma} +  u_s^{n+2-2\gamma} \rho_s^{2\gamma-2} ||\rho_s^{2-2\gamma}R_s^*||_{L^\infty(X)} \big) \\
&\le C(n, \gamma)\int_{B_{2r}^+}  \big( |\nabla \eta|^2 u_s^{n-2\gamma} +  u_s^{n} y^{2\gamma-2}  ||\rho_s^{2-2\gamma}R_s^*||_{L^\infty(X)} \big)  \\
&\le  C(n, \gamma)\int_{B_{2r}^+} \big( \frac{1}{r^2} +   y^{2\gamma-2} ||\rho_s^{2-2\gamma}R_s^*||_{L^\infty(X)} \big).
\end{align*}
This concludes the proof.
\end{proof}


\begin{thebibliography}{100}
\bibitem{A} M. Anderson, Einstein metrics with prescribed conformal infinity on 4-manifolds, Geom. Funct. Anal. 18 (2008), no.2, 305-366.


\bibitem{BH} O. Biquard and M. Herzlich, Analyse sur un demi-espace hyperbolique et poly-homog\'{e}n\'{e}it\'{e} locale, Calc. Var.  P.D.E., 51 (2014), 813-848.


\bibitem{CC} 
J. Case and S.-Y. A. Chang, On fractional GJMS operators, Comm. Pure Appl. Math. 69 (2016), 1017-1061.


\bibitem{CDLS} P. T. Chru\'{s}ciel, E. Delay, J. Lee, and D. N. Skinner,  Boundary regularity of conformally  compact Einstein metrics, J. Differential Geom. 69 (2005), no.1, 111-136.

\bibitem{CG2} S.-Y. A. Chang and M. d. M. Gonz\'{a}lez. Fractional Laplacian in conformal geometry. Adv. 
Math., 226(2):1410–1432, 2011.

\bibitem{CG} S.-Y. A. Chang and Y. Ge, Compactness of conformally compact Einstein manifolds in dimension 4,
Adv. Math. 340 (2018), 588-652.

\bibitem{CGJQ} S.-Y. A. Chang, Y. Ge, S. Jin and J. Qing, On compactness conformally compact Einstein manifolds
and uniqueness of Graham-Lee metrics, III, arXiv:2107.03075.


\bibitem{CGQ} S.-Y. A. Chang and Y. Ge,  Compactness of conformally compact Einstein 4-manifolds II,
Adv. Math. 373 (2020), 107325.

\bibitem{CLW} X. Chen, M. Lai, and F. Wang, Escobar-Yamabe compactifications for Poincar\`{e}-Einstein manifolds  and rigidity theorems, Adv. Math. 343 (2019), 16-35.


\bibitem{CS}  L. Caffarelli and L. Silvestre. An extension problem related to the fractional Laplacian. Communications in partial
differential equations 32 (2007), no. 8, 1245-1260.


\bibitem{DM} S. Dutta and M. Javaheri, Rigidity of conformally compact manifolds with the round sphere as conformal infinity, Adv. Math. 224, (2010), 525-538.



\bibitem{FG} C. Fefferman, and R.Graham, Q-curvature and Poincar\`{e} metrics, Math.
Res. Lett., 9 (2002), no. 2-3, 139-151.

\bibitem{FG2} C. Fefferman and C. R. Graham, The ambient metric, Annals of Mathematics Studies, 178. Princeton University Press, Princeton,  (2012).


\bibitem{GH} M. J. Gursky and Q. Han, Non-existence of Poincar\`{e}-Einstein manifolds with prescribed conformal infinity,  Geom. Funct. Anal. 27 (2017), no.4, 863-879.

\bibitem{GHS} M. J. Gursky, Q. Han, and S. Stolz, An invariant related to the existence of conformally compact Einstein  fillings, Trans. Amer. Math. Soc., 374 (2021), 4185-4205.

\bibitem{GL} C. R. Graham and J. Lee, Einstein metrics with prescribed conformal infinity on the ball, Adv. Math., 87 (1991), 186-225.

\bibitem{GQ} C. Guillarmou and J. Qing, Spectral characterization of Poincar´e-Einstein manifolds with infinity of positive
Yamabe type, Int. Math. Res. Not. (2010), 1720-1740.

\bibitem{GS} M. J. Gursky and G. Sz\'{e}kelyhidi A local existence result for Poincar\`{e}-Einstein metrics, Adv. Math. 361 (2020) 106–912.

\bibitem{GZ}
C. R. Graham and M. Zworski, Scattering matrix in conformal geometry, Invent. Math. 152 (2003), 89-118.

\bibitem{HL} Q. Han and F. Lin, Elliptic Partial Differential Equations, volume1. American Mathematical Soc., 2011


\bibitem{L}
J. Lee, The spectrum of an asymptotically hyperbolic Einstein manifold. Comm. Anal. Geom. 3 (1995), no. 1-2, 253–271.


\bibitem{LSQ} G. Li, J. Qing, and Y.Shi, Gap phenomena and curvature estimates for conformally compact  Einstein manifolds, Trans. Amer. Math. Soc., 369 (2017), 4385-4413.


\bibitem{MM} R. Mazzeo and R. Melrose, Meromorphic extension of the resolvent on complete spaces with asymptotically constant negative curvature, J. Funct. Anal. 113 (1991), 25-45.

\bibitem{STV1} Y. Sire, S. Terracini, and S. Vita, Liouville type theorems and regularity of solutions to degenerate or singular
problems part I: even solutions. Comm. Partial Differential Equations 46-2 (2021), 310-361.

\bibitem{STV2}  Y. Sire, S. Terracini, and S. Vita, Liouville type theorems and regularity of solutions to degenerate or singular
problems part II: odd solutions. Math. Eng. 3-1 (2021), 1-50.

\bibitem{WZ}
F. Wang and H. Zhou, A note on the compactness of Poincar\`{e}-Einstein manifolds, arXiv:2106.01704, 2021.

\bibitem{Y}
R. Yang, On higher order extensions for the fractional Laplacian, arXiv:1302.4413, preprint.

























\end{thebibliography}
\end{document}